\documentclass[11pt]{amsart}
\usepackage{amssymb, amsmath, amsthm}
\usepackage{graphicx}
\usepackage[final, hypertex]{hyperref}
\usepackage[a4paper, centering]{geometry}
\newcommand{\picturename}[1]{#1.eps}
\newtheorem{teor}{Theorem}[section]
\newtheorem{prop}[teor]{Proposition}
\newtheorem{lemma}[teor]{Lemma}
\newtheorem{cor}[teor]{Corollary}
\theoremstyle{definition}
\newtheorem{dhef}[teor]{Definition}
\theoremstyle{remark}
\newtheorem{rk}[teor]{Remark}
\newtheorem{ehse}[teor]{Example}

\long\def\elimina#1{} 

\def\R{\mathbb{R}}

\def\N{\mathbb{N}}

\def\pscal#1#2{\left\langle#1,\,#2\right\rangle}

\def\gauge{\rho}

\def\pgauge{\gauge^0}

\def\len{l}

\def\vf{v_f}

\def\uo{u_{\phi}}
\def\wo{w_{\phi}}
\def\uf{u_{f}}

\def\pscal#1#2{\langle#1,\,#2\rangle}

\def\oray#1{{]\!]#1[\![}}
\def\cray#1{{[\![#1]\!]}}

\def\ccurve#1{{\Gamma}_{#1}}
\def\len#1{L(#1)}
\def\spaceX{X_{\phi}}

\def\curveD#1{\Gamma^D_{#1}}
\def\up{u^{\psi}}
\def\wp{w^{\psi}}

\DeclareMathOperator{\inte}{int} \DeclareMathOperator{\spt}{supp}
 \DeclareMathOperator{\dive}{div}
 
\DeclareMathOperator{\proj}{\Pi}

\begin{document}

\title[A nonhomogeneous boundary value problem]
{A nonhomogeneous boundary value problem \\ in mass transfer theory}%
\author[G.~Crasta]{Graziano Crasta}
\address{Dipartimento di Matematica ``G.\ Castelnuovo'', Univ.\ di Roma I\\
P.le A.\ Moro 2 -- 00185 Roma (Italy)}
\email[Graziano Crasta]{crasta@mat.uniroma1.it}
\author[A.~Malusa]{Annalisa Malusa}
\email[Annalisa Malusa]{malusa@mat.uniroma1.it}
\keywords{Boundary value problems,
mass transfer theory}
\subjclass[2010]{Primary 35A02, Secondary 35J25}


\begin{abstract}
We prove a uniqueness result of solutions for a system of PDEs of
Monge-Kantorovich type arising in problems of mass transfer theory.
The results are obtained under very mild regularity assumptions both
on the reference set $\Omega\subset \R^n$, and on the (possibly asymmetric) norm
defined in $\Omega$. In the special case when $\Omega$ is endowed with the Euclidean
metric, our results provide a complete description of the stationary
solutions to the tray table problem in granular matter theory.
\end{abstract}

\maketitle

\section{Introduction}

The model system usually considered for the description of the stationary
configurations of sandpiles on a tray table is the Monge-Kantorovich type
system of PDEs
\[
\begin{cases}
-\dive(v\, Du) = f & \text{in $\Omega$},\\
|Du|\leq 1,\ v\geq 0 & \text{in $\Omega$},\\
(1-|Du|)v=0 & \text{in $\Omega$},\\
u=\phi & \text{on $\partial\Omega$}
\end{cases}
\]
(see, e.g., \cite{AEW,CaCa,HK}).
The data of the problem are the flat surface of the table $\Omega\subseteq \R^2$, the profile
of the tray table $\phi$, and the density of the source $f\geq 0$.
The dynamical behavior of the granular matter is pictured by the pair
$(u,v)$, where $u$ is the profile of the standing layer, whose slope has not to exceed a critical value
($|Du|\leq 1$)
in order to prevent avalanches, while $v\geq 0$ is the thickness of the rolling layer
(see Figure~\ref{fig:sand}).
The condition $(1-|Du|)v=0$ corresponds to require that
the matter runs down only in the region where the slope of the heaps is maximal.
On the border of the table the standing layer fills the gap with the height $\phi$, and the exceeding sand falls down.
If $\phi=0$, the model reduces to the open table problem, largely investigated in recent years
(see e.g.\ \cite{CaCa,CCCG,CCS,Pr} and the references therein).

\begin{figure}
\includegraphics[height=4cm]{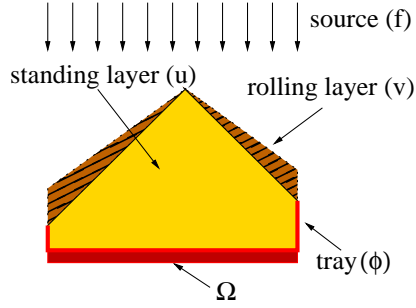}\qquad\qquad
\caption{The tray table problem}
\label{fig:sand}
\end{figure}

In the present paper we consider a more general system of PDEs in the open,
bounded and connected set $\Omega\subseteq \R^n$
\begin{equation}\label{f:MKint}
\begin{cases}
-\dive(v\, D\gauge(Du)) = f
&\text{in $\Omega$},\\
\gauge(Du)\leq 1,\ v\geq 0
&\text{in $\Omega$},\\
(1-\gauge(Du))v=0 &\text{in $\Omega$},\\
u=\phi &\text{on $\partial\Omega$}\,,
\end{cases}
\end{equation}
in the unknowns $v\in L^1_+(\Omega)$, $u\in W^{1,\infty}(\Omega)\cap C(\overline{\Omega})$.
The data are:
\begin{itemize}
\item[-] $K\subseteq \R^n$, a compact convex set of class $C^1$
containing the origin in its interior, and its (strictly convex) polar set $K^0$.
The set $K$ appears in (\ref{f:MKint}) by means of its gauge function $\gauge \colon \R^n \to [0,+\infty)$.
Moreover the closure of 
$\Omega$ will be equipped with a (possibly asymmetric) geodesic distance $d_L$ given in terms of the gauge function $\pgauge$ of the polar set $K^0$
(see Section~\ref{s:extens});
\item[-] $f\in L^1_+(\Omega)$; 
\item[-] $\phi\colon \partial\Omega \to \R^n$, $1$-Lipschitz function w.r.t.\ $d_L$.
\end{itemize}
This general version takes into account the possibility of homogeneous aniso\-tropies,
and can be applied, for example, to the Bean's model for the description of the
macroscopic electrodynamics
of hard superconductors (see \cite{CMf,CMg,CMi,CMh},  for the
case $\phi=0$ and $\Omega$, $K$ smooth sets),
and to optimal shape design (see \cite{BoBu}).

S.~Bianchini in \cite{Bi} proved that
the Lax--Hopf function $\uo$ (defined in (\ref{f:LH}) below),
coupled with an explicit $\vf\in L^1_+(\Omega)$ solves (\ref{f:MKint})
(see Theorem~\ref{f:bian}).
The function $\vf$ is constructed as follows. The set $\Omega$ can be covered by a
family of disjoint transport rays, and every $v$-component of a solution to (\ref{f:MKint})
satisfies a first order linear ODE along almost every ray.
The function $\vf$ is then uniquely defined along each ray, by requiring that it vanishes
at the final point. This procedure gives a solution to (\ref{f:MKint}), thanks to a
disintegration formula for the Lebesgue measure along the rays.

In this paper we shall provide a complete characterization of the solutions $(u,v)$
to (\ref{f:MKint}).
Starting from the fact that $(\uo, \vf)$ is a solution, we show that
a pair $(u,v)$ is a solution if and only if both $(u, \vf)$ and $(\uo, v)$ are
solutions, so that we can decouple the original problem (see Section~\ref{s:minimum}).
Moreover, the admissible $u$-components are characterized as the solutions of a minimization problem
which has an explicit minimal solution $\uf$ depending on the source term $f$.
Since $\uo$ is the maximal solution of the problem, the uniqueness
of the $u$-component corresponds to the case when $\uf=\uo$. We prove that this can occur
if and only if the final points in $\Omega$ of the transport rays are contained in the support
of the source $f$ (see Section \ref{s:uniqueu}).
Concerning the $v$-component, in \cite{Bi} it is proved that every admissible $v$ has to vanish
at the final points in $\Omega$ of the transport rays. Hence, by the very definition of $\vf$,
it is the unique $v$-component if the union $T$ of the rays with both endpoints on $\partial \Omega$
has null Lebesgue measure.
We shall prove that the converse
is also true: if the set $T$ has positive Lebesgue measure, we can construct
a nontrivial function $v\in L^1_+(\Omega)$ such that $-\dive(v\, D\gauge(D\uo)) = 0$, so that the uniqueness
of the $v$-component fails.

These results generalize those in \cite{CMf,CMg,CMi} to the case when
the reference set $\Omega$ is possibly non regular, the boundary datum $\phi$ is possibly
non-homogeneous, and the constraint $K$ is a possibly non-strictly convex set.
In this case we have to use a completely different strategy with respect to the
previous approaches.

Namely, if $\Omega$ and $K$ are smooth enough, then it is
possible to define a suitable notion of  (anisotropic) curvatures of $\partial \Omega$,
and the analysis of solutions to (\ref{f:MKint}) can be performed using the methods
of Differential Geometry (see, e.g.,
\cite{CaMS,CMf,CMg,CMi,IT,LN,MM}).
In the non-smooth case these methods do not work, and a different approach is needed, based on the
techniques developed in Optimal Transport Theory (see \cite{Amb,Bi, Vil}).
Namely, system (\ref{f:MKint}) can be understood as the PDE formulation of an optimal transport
problem with strictly convex cost function $\pgauge$, assigned initial distribution $\mu=f\, dx$ and
transport potential given by the Lax-Hopf function $\uo$ (so that the final mass distribution $\nu$
is the measure concentrated on $\partial \Omega$ induced by the transportation of $\mu$ along the rays
associated to $\uo$). From this point of view, the minimization problem considered in Section~\ref{s:minimum}
turns out to be the dual formulation of the Monge problem (see, e.g., \cite{Vil}
for an exhaustive presentation of this subject).

A comment on the regularity assumption on $K$ is in order.
Namely, if $K$ is of class $C^1$, then its polar set $K^0$ need not be of class $C^1$,
but it is strictly convex.
It is well-known that the strict convexity of $K^0$, which is the unit ball of the
dual norm $\pgauge$, allows non-trivial simplifications in the proofs;
for example, the direction of the transport rays coincides with $D\gauge(Du)$
almost everywhere in $\Omega$.
Related results in the case of a general possibly non-regular constraint $K$
can be found in \cite{BiGl}.

The paper is organized as follows. In Section~\ref{s:prel} we collect all the notation and preliminary results.
In Section~\ref{s:extens}
we recall some generalization of the McShane's extension theorem for Lipschitz functions with
constrained gradient. These general results will be used for introducing the maximal and the
minimal extension of the boundary datum $\phi$ consistent with the gradient constraint.
Section~\ref{s:existence} is devoted to
the statement of the problem, some remarks on the hypotheses,
and the statement of the existence result due to S.~Bianchini
(see \cite[Thm.~7.1]{Bi}).
One of the main tools needed for proving the uniqueness results is the fact that (\ref{f:MKint}) is the Euler-Lagrange
necessary condition for a minimization problem with a gradient constraint.
This property is proved in Section~\ref{s:minimum}.
The necessary and sufficient conditions for the uniqueness of the solutions $(u,v)$ to (\ref{f:MKint}) are proved
in Sections~\ref{s:unique} and~\ref{s:uniqueu}.
As a byproduct of our analysis, with a little additional effort, we provide a
complete description of the singular sets considered in this kind of optimal transport problems.

\bigskip
\textbf{Acknowledgements.}
The authors wish to thank the anonymous referee
for valuable comments that have improved the presentation of the manuscript.

\section{Notation and preliminaries}
\label{s:prel}

The standard scalar product of $x,y\in\R^n$
will be denoted by $\pscal{x}{y}$, while $|x|$ will denote the
Euclidean norm of $x$.
Concerning the segment joining $x$ with $y$, we set
\[
\cray{x,y} := \{tx+(1-t)y;\ t\in [0,1]\},
\qquad \oray{x,y} := \cray{x,y}\setminus\{x,y\}.
\]

Given a set $A\subset \R^n$, its interior, its closure and its boundary
will be denoted by $\inte A$, $\overline{A}$ and $\partial A$ respectively.

A bounded open set $\mathcal{O}\subset\R^n$
(or, equivalently,
$\overline{\mathcal{O}}$ or $\partial\mathcal{O}$)
is of class $C^k$, $k\in\N$,
if for every point $x_0\in\partial \mathcal{O}$
there exists a ball $B=B_r(x_0)$ and a one-to-one
mapping $\psi\colon B\to D$ such that
$\psi\in C^k(B)$, $\psi^{-1}\in C^k(D)$,
$\psi(B\cap \mathcal{O})\subseteq\{x\in\R^n;\ x_n > 0\}$,
$\psi(B\cap\partial \mathcal{O})\subseteq\{x\in\R^n;\ x_n = 0\}$.

Let us now fix the notation and the basic results concerning
the convex set which plays the role of gradient constraint
for the $u$-component in (\ref{f:MKint}).
In the following we shall assume that
\begin{equation}\label{f:hypk}
\text{$K$ is a compact, convex subset of $\R^n$ of class $C^1$, with $0\in\inte K$.}
\end{equation}
Let us denote by $K^0$ the
polar set of $K$, that is
\[
K^0 := \{p\in\R^n;\ \pscal{p}{x}\leq 1\ \forall x\in K\}\,.
\]
We recall that, if $K$ satisfies (\ref{f:hypk}), then
$K^0$ is a compact, strictly convex subset of $\R^n$
containing the origin in its interior, and $K^{00} = (K^0)^0 = K$
(see, e.g., \cite{Sch}).

The gauge function $\gauge\colon \R^n \to \R$ of $K$ is defined by
\[
\gauge(\xi) := \inf\{ t\geq 0;\ \xi\in t K\}=\max \{\pscal{\xi}{\eta},\ \eta\in K^0\}\,,
\quad \xi\in\R^n\,.
\]
It is straightforward to see that $\gauge$ is a positively 1-homogeneous convex function such that
$K=\{\xi\in\R^n\colon\ \gauge(\xi)\leq 1\}$. The gauge function of the set $K^0$ will be denoted by $\pgauge$.

The properties of the gauge functions needed in the paper are collected in the following
theorem.

\begin{teor}\label{t:sch}
Assume that $K\subseteq \R^n$ satisfies (\ref{f:hypk}).
Then the following hold:

\par\noindent (i) $\gauge$ is continuously
differentiable in $\R^n\setminus \{0\}$, and
\[
\gauge(\xi+\eta)\leq \gauge(\xi) + \gauge(\eta)\, \quad \forall\ \xi,\eta\in\R^n\,.
\]

\par\noindent (ii) $K^0$ is strictly convex,  and
\[
\begin{split}
\pgauge(\xi+\eta) & \leq \pgauge(\xi) + \pgauge(\eta),  \quad \forall\ \xi,\eta\in\R^n \\
\pgauge(\xi+\eta) &= \pgauge(\xi) + \pgauge(\eta)\ \Leftrightarrow\
\exists \lambda\geq 0\ \colon\ \xi = \lambda\, \eta\ \text{or}\ \eta = \lambda\, \xi.
\end{split}
\]

\par\noindent (iii)
For every $\xi\neq 0$,
$D\gauge(\xi)$ belongs to $\partial K^0$, and
\[
\pscal{D\gauge(\xi)}{\xi} = \gauge(\xi),\quad
\pscal{p}{\xi} < \gauge(\xi)\ \forall p\in K^0,\ p\neq D\gauge(\xi)\,.
\]
\end{teor}

\begin{proof}
See \cite{Sch}, Section 1.7.
\end{proof}

In what follows we shall consider $\R^n$ endowed with the possibly asymmetric norm $\pgauge(x-y)$,
$x,y\in\R^n$.
By Theorem~\ref{t:sch}(ii), the unit ball $K^0$ of $\pgauge$ is strictly convex but,
under the sole assumption (\ref{f:hypk}), it need not be differentiable.
Moreover, the Minkowski structure $(\R^n, \pgauge)$
is not a metric space in the usual sense, since $\pgauge$
need not be symmetric (for an introduction to nonsymmetric metrics see \cite{Gromov}).
Finally, since
$K^0$ is compact and $0\in \inte K^0$, then the convex metric is equivalent to the Euclidean one,
that is there exist $c_1$, $c_2>0$ such that $c_1|\xi|\leq \pgauge(\xi)\leq c_2 |\xi|$ for every $\xi\in\R^n$.

The main motivation for introducing the convex metric associated to $\pgauge$ is the fact that the
Sobolev functions with the gradient constrained to belong to $K$ are
the locally 1-Lipschitz functions with respect to $\pgauge$, as stated in the following result.

\begin{lemma}\label{l:diseqrho}
Let $\mathcal{O}\subset\R^n$ be a nonempty open set, and assume that the set $K\subset\R^n$ satisfies (\ref{f:hypk}).
Let $\pgauge$ be the gauge function of $K^0$. Then the following properties are equivalent.
\begin{itemize}
\item[i)] $u\colon \mathcal{O} \to \R$  is a locally $1$-Lipschitz function with respect to $\pgauge$, i.e.
\begin{equation}\label{f:lipur}
u(x_2)-u(x_1)\leq\pgauge(x_2-x_1)\quad \text{for every}\
\cray{x_1,x_2}\subset\mathcal{O}.
\end{equation}
\item[ii)]
$u\in W^{1,\infty}(\mathcal{O})$, and $Du(x) \in K$ for a.e.\ $x\in\mathcal{O}$.
\end{itemize}
\end{lemma}

\begin{proof}
Let $u$ be a locally 1-Lipschitz function in $\mathcal{O}$
with respect to the metric $\pgauge$.
Then $u$ is locally Lipschitz in $\mathcal{O}$,
and if $x_0$ is a differentiability point of $u$, by (\ref{f:lipur}) we have
\[
\begin{split}
o(|t|) & =\frac{u(x_0+t\xi)-u(x_0)-t \pscal{Du(x_0)}{\xi}}{t} \leq \pgauge(\xi)-\pscal{Du(x_0)}{\xi} \\
& \leq 1-\pscal{Du(x_0)}{\xi}\,, \qquad \forall \xi \in K^0.
\end{split}
\]
Hence $\pscal{Du(x_0)}{\xi}\leq 1 $ for all $\xi \in K^0$, that is $Du(x_0)\in (K^0)^0=K$.
Since $u$ is a locally Lipschitz function in $\mathcal{O}$ with gradient
in the bounded set $K$, (ii) holds.

Conversely, let $u$ satisfy (ii), and let
$\cray{x_1,x_2}\subset\mathcal{O}$.
Clearly, we can assume $x_1\neq x_2$.
Let $S_r := \{z\in B_r(0):\ \pscal{z-x_1}{x_2-x_1} = 0\}$, and let
$Q_r := \left\{\bigcup \cray{x_1+z, x_2+z}:\ z\in S_r\right\}$
be the cylinder of radius $r$ around the segment $\cray{x_1,x_2}$.
Let us choose $r>0$ such that $Q_r\subset \Omega$.
By (ii) there exists a set $N\subset S_r$ with $(n-1)$-Lebesgue measure zero,
such that
for every $z\in S_r\setminus N$,
$u$ is absolutely continuous along
the segment $\cray{x_1+z, x_2+z}$,
and
$Du(x_1+z+s(x_2-x_1))\in K$ for a.e.\ $s\in [0,1]$.
Let $(z_j)\subset S_r\setminus N$, $z_j\to 0$. Then, for every $j\in\N$,
\[
u(x_2+z_j) - u(x_1+z_j) = \int_0^1 \pscal{Du(x_1+z_j +s(x_2-x_1))}{x_2-x_1}\, ds
\leq \pgauge(x_2-x_1).
\]
The conclusion now follows from the continuity of $u$ in $\mathcal{O}$, passing to the limit in $j$.
\end{proof}

In the sequel we need to consider locally 1-Lipschitz functions with respect to $\pgauge$ that agree with
a given continuous boundary datum. This can be done only if the boundary values are compatible with
the gradient bounds, and, in absence of assumptions on the regularity of the reference set,
the compatibility condition concerns the variations of the data along every admissible path.

Here we introduce some notation for the curves, while the compatibility condition will be discussed in Section~\ref{s:extens}.

\begin{dhef}\label{d:gamma}
Let $D\subset\R^n$ be a nonempty compact set.
For every $x_1, x_2\in D$,
$\curveD{x_1, x_2}$ will denote the (possibly empty)
family of absolutely continuous curves $\gamma\colon [0,1]\to D$
such that $\gamma(0) = x_1$ and $\gamma(1) = x_2$.
\end{dhef}

For every absolutely continuous curve $\gamma\colon [0,1]\to \R^n$, let us denote by
$\len{\gamma}$ its length with respect to the convex metric associated to $\pgauge$, that is
\[
\len{\gamma} := \int_0^1 \pgauge(\gamma'(t))\, dt\,.
\]

If $D$ is a compact subset of $\R^n$,
by a standard compactness argument we have that
if $\curveD{x_1,x_2}\neq\emptyset$, then
there exists a geodesic in $D$
joining $x_1$ to $x_2$, i.e.\ there exists a curve $\tilde{\gamma}\in \curveD{x_1, x_2}$
such that
$\len{\tilde{\gamma}}\leq \len{\gamma}$ for every $\gamma\in \curveD{x_1, x_2}$
(see e.g.\ \cite[Thm.~4.3.2]{AmTi}, \cite[\S14.1]{Ces}).

For sake of completeness we briefly recall that any geodesic curve lying in the interior of $D$
is in fact a segment, due to strict convexity of $\pgauge$.

\begin{lemma}\label{l:geoint}
Let $D\subset\R^n$ be a nonempty compact set.
Let $x_1, x_2\in D$ be two points such that $\curveD{x_1, x_2}\neq\emptyset$,
and let $\gamma\in\curveD{x_1, x_2}$ be a geodesic.
If $\gamma(t)\in\inte D$ for every $t\in (0,1)$, then the support of $\gamma$ is the
segment $\cray{x_1, x_2}$.
\end{lemma}

\begin{proof}
By an approximation argument we can assume that $x_1, x_2\in\inte{D}$ so that
$\gamma(t)\in \inte{D}$ for every $t\in [0,1]$.
Let $\eta(t) = x_1 + t(x_2-x_1)$, $t\in [0,1]$, be a parametrization of the segment
$\cray{x_1, x_2}$.
Assume by contradiction that the support of $\gamma$ is not the segment $\cray{x_1, x_2}$.
Since $K^0$ is strictly convex we have that
$\len{\gamma} > \len{\eta} = \pgauge(x_2-x_1)$.
Since the support of $\gamma$ is contained in $\inte{D}$, there exists
$\varepsilon \in (0,1)$ such that the support of
$\gamma_{\varepsilon} := \gamma + \varepsilon(\eta-\gamma)$
is contained in $\inte{D}$.
Then
\[
\begin{split}
\len{\gamma_{\varepsilon}} & = \int_0^1 \pgauge((1-\varepsilon)\gamma' + \varepsilon\eta')
\leq \int_0^1 [(1-\varepsilon)\pgauge(\gamma') + \varepsilon\pgauge(\eta')]
\\ & = (1-\varepsilon)\len{\gamma} + \varepsilon\len{\eta} < \len{\gamma},
\end{split}
\]
in contradiction with the minimality of $\gamma$.
\end{proof}

\section{Lipschitz extensions}\label{s:extens}

In this section, $C$ and $D$ will be two sets such that
\begin{equation}\label{f:CD}
\begin{split}
& \text{$D$ is a nonempty compact subset of $\R^n$, and}\\
& \text{$C$ is a closed subset of $D$ containing $\partial D$,}
\end{split}
\end{equation}
so that $D\setminus C$ is an open (possibly empty) set.

For a given continuous function $\psi\colon C\to \R$, we want to discuss the existence of a
locally Lipschitz extension of $\psi$ in $D$ with the gradient constrained to belong to a
convex set $K$.

It is well known that, in the case $D=\overline{\Omega}$, $\Omega$ bounded open subset of $\R^n$,
$C=\partial \Omega$, and $K$ satisfying (\ref{f:hypk}), a necessary and sufficient condition
for the existence of such an extension is
\[
\psi(x) - \psi(y) \leq \len{\gamma}\qquad
\forall x,y\in \partial\Omega,\ \forall \gamma\in\Gamma^\Omega_{y,x}
\]
where $\Gamma^\Omega_{y,x}$ is the family of absolutely continuous curves $\gamma\colon [0,1]\to \overline{\Omega}$
such that $\gamma(t)\in\Omega$ for every $t\in(0,1)$, $\gamma(0) = y$ and $\gamma(1) = x$
(see \cite[Sect.~5]{Li} and
\cite{Ar-b}). Moreover, if $\Omega$ has a Lipschitz boundary,
one can relax the above necessary and sufficient condition by taking paths in $\Gamma^{\overline{\Omega}}_{y,x}$
(see \cite{Li}).

In order to handle more general situations, we will consider
a continuous function $\psi\colon C\to \R$ satisfying
\begin{equation}\label{f:H4}
\psi(x) - \psi(y) \leq \len{\gamma}\qquad
\forall x,y\in C,\ \forall \gamma\in\curveD{y, x}\,,
\end{equation}
where $\curveD{x_1,x_2}$ is the family of curves introduced in Definition~\ref{d:gamma}.

This condition can be rephrased in the setting of the length space $(D, d_L)$,
where $d_L\colon D\times D \to [0,+\infty]$ is the (possibly asymmetric) distance defined by
\begin{equation}\label{f:dL}
d_L(x,y) := \inf\{\len{\gamma}:\ \gamma\in\curveD{y, x}\},
\qquad x,y\in D.
\end{equation}
(We refer to \cite[Chap.~2]{BuBuIv} for the basic properties of length spaces; see in particular
Remark~2.2.6 for the extension to possibly asymmetric metrics.)
Namely, (\ref{f:H4}) is equivalent to requiring that
$\psi$ is a $1$-Lipschitz function with respect to the $d_L$.

In the following lemma we recall the well known fact that
(\ref{f:H4}) turns out to be a sufficient condition for the existence
of minimal and maximal $1$-Lipschitz extensions (w.r.t. $d_L$)  of $\psi$ . Notice that (\ref{f:H4})
is no more a necessary condition, as it can be seen when $D$ is a closed segment in
$\R^2$.

\begin{lemma}\label{l:genprop}
Let $C$, $D$ satisfy (\ref{f:CD}),
let $K$ satisfy (\ref{f:hypk}),
and let
$\psi\colon C\to\R$ be a continuous function satisfying (\ref{f:H4}).
Then the functions
\begin{equation}\label{f:psi}
\begin{split}
\up(x) & := \inf\left\{\psi(y)+L(\gamma):\ y\in C,\ \gamma\in\curveD{y,x}
\right\},\quad
x\in D,\\
\wp(x) & := \sup\left\{\psi(y)-L(\gamma):\ y\in C,\ \gamma\in\curveD{x,y}
\right\},\quad
x\in D,
\end{split}
\end{equation}
are continuous in $D$, and $\up=\wp=\psi$ in $C$.
Moreover, $\up, \wp\in W^{1,\infty}(\inte D)$
and $D\up, D\wp\in K$ a.e.\ in $\inte D$.
\end{lemma}

As a consequence of Lemma~\ref{l:genprop},
the functions $\up$ and $\wp$ belong to the functional space
\[
Z_\psi :=\left\{u\in C(D)\cap W^{1,\infty}(\inte D):\
u=\psi\ \text{on}\ C,\ Du\in K\ \text{a.e.\ in}\ \inte D \right\}.
\]
It should be noted that in the definition of $\up(x)$ the infimum is taken over all paths
in $D$ joining $y\in C$ to $x$, while the supremum in the definition of $\wp$ is taken
over all paths in $D$ joining $x$ to $y\in C$.
This asymmetry is needed in order to construct exactly the
maximal and the minimal elements of $Z_\psi$, as it is stated in Lemma \ref{l:vischar} below
(see \cite[Prop.~2.1]{Bi}).

\begin{dhef}\label{d:visible}
Given $x\in D$, we shall say that $y\in C$ is a visible point for $x$ if the segment $\oray{y,x}$
is entirely contained in the open set $D\setminus C$. We shall denote by $V^C_x$ the set of all
visible points for $x$, that is $V_x^C = \{y\in C:\ \oray{y,x}\subset D\setminus C\}$.
\end{dhef}

\begin{lemma}\label{l:vischar}
The functions $\up$, $\wp$ defined in (\ref{f:psi}) are characterized in the following way:
\begin{equation}\label{f:LH3}
\begin{split}
\up(x) & = \min\{\psi(y) + \pgauge(x-y):\ y\in V_x^C\}, \qquad x\in D,\\
\wp(x) & = \max\{\psi(y) - \pgauge(y-x):\ y\in V_x^C\}, \qquad x\in D.
\end{split}
\end{equation}
Moreover $\wp\leq u\leq \up$ for every $u\in Z_\psi$.
\end{lemma}

\begin{cor}\label{c:penduno}
Let $x\in D\setminus C$ be a  point of differentiability for $\up$ (resp.\ $\wp$). Then
$D\up(x)$ (resp.\ $D\wp(x)$) belongs to $\partial K$.
\end{cor}

\begin{proof}
Given $x\in D\setminus C$, let $y\in V_x^C$ be such that
\[
\up(x)=\psi(y)+\pgauge(x-y)\,.
\]
Setting $\xi=(x-y)/\pgauge(x-y)$, the fact that $\up\in Z_{\psi}$ and Lemma \ref{l:diseqrho} imply that
\[
\up(x-s\xi)-\up(y)  \leq \pgauge (x-y)-s\,, \
\up(x)-\up(x-s\xi)\leq s \quad \forall s\in[0, \pgauge(x-y)]\,.
\]
Then we have
\[
-s \leq \up(x-s\xi)-\up(x)=\up(x-s\xi)-\up(y)-\pgauge(x-y)\leq -s\,,
\]
and hence
\begin{equation}\label{f:linpsi}
\up(x-s\xi)=\up(x)-s,\qquad \forall s\in[0, \pgauge(x-y)]\,.
\end{equation}
If $\up$ is differentiable at $x$, (\ref{f:linpsi}) implies that $\pscal{D\up(x)}{\xi}=1$, so that
\[
\gauge(D\up(x))=\max_{\eta\in K^0}\pscal{D\up(x)}{\eta}\geq 1\,.
\]
On the other hand, by Lemma~\ref{l:genprop},
$D\up(x)\in K$,
hence $D\up(x)$ belongs to $\partial K$.

The proof for $\wp$ is omitted, since it is entirely similar.
\end{proof}

\section{The existence result}
\label{s:existence}

Our main goal is to provide a full analysis of the Monge-Kantorovich system
of PDEs
\begin{equation}\label{f:MK}
\begin{cases}
-\dive(v\, D\gauge(Du)) = f
&\text{in $\Omega$},\\
\gauge(Du)\leq 1,\ v\geq 0
&\text{a.e.\ in $\Omega$},\\
(1-\gauge(Du))v=0 &\text{a.e.\ in $\Omega$},\\
u=\phi &\text{pointwise on $\partial\Omega$}\,,
\end{cases}
\end{equation}
under the following assumptions:
\begin{itemize}
\item[(H1)] $\Omega$ is a nonempty bounded connected open subset of $\R^n$, $n\geq 1$.
\item[(H2)] $f\in L^1(\Omega)$, $f\geq 0$ a.e.\ in $\Omega$.
\item[(H3)] $\gauge\colon \R^n \to \R$ is the gauge function of a set $K\subset\R^n$ satisfying (\ref{f:hypk}).
\item[(H4)] $\phi\colon\partial\Omega\to \R$ is a continuous function, satisfying the condition
\[
\phi(x)-\phi(y) \leq \len{\gamma} := \int_0^1 \pgauge(\gamma'(t))\, dt\qquad \forall x,y\in\partial\Omega,\
\forall \gamma\in \ccurve{y,x},
\]
where
\[
\ccurve{y,x}:= \ccurve{y,x}^{\overline{\Omega}} = \left\{
\gamma\in AC([0,1], \overline{\Omega}):\ \gamma(0) = y,\ \gamma(1) = x
\right\}\,.
\]
\end{itemize}
\begin{rk}
In the first equation in (\ref{f:MK}) we understand $D\gauge$
defined also at the origin since, by the third condition,
$v=0$ a.e.\ where $Du=0$.
\end{rk}
The functional setting for the unknowns $(u,v)$ in (\ref{f:MK}) is $\spaceX \times L^1_+(\Omega)$, where
\begin{equation}\label{f:X}
\spaceX := \left\{
u\in C(\overline{\Omega})\cap W^{1,\infty}(\Omega):\
u=\phi \text{ on } \partial\Omega,\ Du\in K \text{ a.e.~in } \Omega
\right\}
\end{equation}
and
\[
L^1_+(\Omega) := \{v\in L^1(\Omega):\ v\geq 0 \text{ a.e.\ in } \Omega\}\,.
\]

By a solution of (\ref{f:MK})
we mean a pair $(u,v)\in \spaceX\times L^1_+(\Omega)$ satisfying
\begin{equation}\label{f:weak}
\int_{\Omega} v \pscal{D\gauge(Du)}{D\varphi} = \int_{\Omega} f\varphi
\qquad \forall \varphi\in C^{\infty}_0(\Omega)
\end{equation}
and $(1-\gauge(Du))v=0$ a.e.\ in $\Omega$.

Let us define the Lax-Hopf function
$\uo\colon\overline{\Omega}\to\R$ by
\begin{equation}\label{f:LH}
\uo(x) := \inf\left\{\phi(y)+L(\gamma):\ y\in\partial\Omega,\ \gamma\in\ccurve{y,x}
\right\},\quad
x\in \overline{\Omega},
\end{equation}
which corresponds to the function $\up$ defined in (\ref{f:LH3}) in the
special case $D=\overline{\Omega}$,
$C = \partial\Omega$, and
$\psi = \phi$
(observe that $C=\partial \Omega\supseteq\partial \overline{\Omega}=\partial D$
and $\Omega\subseteq\inte D$).

As a consequence of Lemmas~\ref{l:genprop}, \ref{l:vischar}, and Corollary~\ref{c:penduno},
we obtain
the following result, which is well-known when $\Omega$ is a Lipschitz domain
(see e.g.\ \cite[Chap.~5]{Li}).

\begin{teor}\label{t:genprop}
If (H1), (H3) and (H4) are fulfilled,
then $\uo\in\spaceX$ and $D\uo\in\partial K$ a.e.\ in $\Omega$.
Moreover, $\uo$ is the maximal element in $\spaceX$, and it
is characterized by
\begin{equation}\label{f:LH2}
\uo(x) = \min\{\phi(y) + \pgauge(x-y):\ y\in V_x\}, \qquad x\in\overline{\Omega},
\end{equation}
where $V_x = \{y\in\partial\Omega:\ \oray{y,x}\subset\Omega\}$.
\end{teor}

If $\phi=0$, the Lax--Hopf function $u_0$
is the distance from $\partial\Omega$ with respect to
the convex metric associated to $\pgauge$,
that is
$u_0(x) = \min\{\pgauge(x-y):\ y\in\partial\Omega\}$, $x\in \Omega$
(see \cite{CMf} for a complete treatment of this subject).

More generally, if
$\phi\colon\partial\Omega\to\R$ satisfies
\begin{equation}\label{f:onelip}
\phi(x)-\phi(y)\leq\pgauge(x-y)
\qquad\forall x,y\in\partial\Omega,
\end{equation}
the Lax--Hopf function $\uo$ coincides with
\begin{equation}\label{f:LH4}
\tilde{u}(x):= \min\{\phi(y)+ \pgauge(x-y):\ y\in\partial\Omega\}\,,\qquad x\in\overline{\Omega}\,.
\end{equation}
Namely, from (\ref{f:onelip}) it is easy to check that $\tilde{u}=\phi=\uo$ on $\partial\Omega$.
Moreover, from the representation formula (\ref{f:LH2}) it is clear that
$\tilde{u} \leq \uo$.
On the other hand,
let $x\in\Omega$ and let $y\in\partial\Omega$
be such that $\tilde{u}(x) = \phi(y) + \pgauge(x-y)$.
Let $z\in\cray{y, x}\cap\partial\Omega$
be such that $\oray{z, x}\subset\Omega$.
Then, by the representation formula (\ref{f:LH2}) and by (\ref{f:onelip})
we have that
\[
\begin{split}
\tilde{u}(x) & \leq \uo(x) \leq \phi(z) + \pgauge(x-z)
\leq \phi(y) + \pgauge(z-y) + \pgauge(x-z)
\\ & = \phi(y) + \pgauge(x-y) = \tilde{u}(x),
\end{split}
\]
so that $\tilde{u}(x) = \uo(x)$.

The following example shows that if we merely require (H4),
then in general $\uo\neq \tilde{u}$.

\begin{ehse}
Given $\epsilon\in (0,\pi)$ let us define the set
\[
\Omega := \{(\rho\cos\theta, \rho\sin\theta)\in\R^2:\
1<\rho<2,\ -\pi < \theta < \pi-\epsilon\}\,,
\]
and let $\phi\colon\partial\Omega\to\R$ be defined by $\phi = \theta$.
If we choose $K = \overline{B}_1(0)$, then $\pgauge(\xi) = |\xi|$ for
every $\xi\in\R^2$, and $\len{\gamma}$ is the usual (Euclidean) length
of a curve $\gamma$.
A straightforward computation shows that
$\phi$ satisfies (H4).
On the other hand, it is also clear that $\phi$ does not satisfy
(\ref{f:onelip}).
Namely, if we consider the two boundary points
$y_1 = (-1,0)$, $y_2 = (\cos(\pi-\epsilon), \sin(\pi-\epsilon))$,
then we have
\[
\phi(y_2) - \phi(y_1) = 2\pi - \epsilon > 2 \sin(\epsilon/2) = |y_2-y_1|.
\]

In this case, the function $\tilde{u}$ defined in (\ref{f:LH4}) does not belong
to $\spaceX$, since $\tilde{u}(y_2)\leq \phi(y_1)+|y_1-y_2|<\phi(y_2)$ and hence it does not achieve the boundary value
at $y_2$.
\end{ehse}


We conclude this section with the following result
due to S. Bianchini (see \cite{Bi} and \cite{BiGl} for its generalization to the case of nonsmooth
constraint sets $K$).

\begin{teor}\label{t:exv}
There exists a weak
solution $\vf\in L^1_+(\Omega)$
to the transport equation
\begin{equation}\label{f:dive}
-\dive(\vf\, D\gauge(D\uo)) = f
\qquad\textrm{in}\ \Omega\,.
\end{equation}
\end{teor}

As a corollary of Theorems~\ref{t:genprop} and~\ref{t:exv}
we conclude that $(\uo, \vf)$ is a solution to (\ref{f:MK}).


\section{The minimum problem}
\label{s:minimum}

In this section we shall discuss the connection between the system of PDEs (\ref{f:MK})
and the
variational problem
\begin{equation}\label{f:min}
\min \left\{-\int_{\Omega}fu\, dx:\ u\in \spaceX\right\}.
\end{equation}
It is clear that, since $f$ is non-negative and $\uo$ is the maximal
element in $\spaceX$, then $\uo$ is a solution to (\ref{f:min}).
We shall show that (\ref{f:MK}) is the Euler-Lagrange necessary and sufficient condition
for the minimum problem (\ref{f:min}).
This property, which is the analogous to the duality in Mass Transport Theory,
will be useful in the proofs of the uniqueness of the
solutions of (\ref{f:MK}).

As a preliminary step, we need to enrich the class of test functions
allowed in (\ref{f:weak}).

\begin{lemma}\label{l:approx}
Let $u\in C(\overline{\Omega})\cap W^{1,\infty}(\Omega)$, $u=0$ on $\partial\Omega$.
Then there exists a sequence $(\psi_j)\subset C^{\infty}_0(\Omega)$ such that
$\psi_j\to u$ and $D\psi_j\to Du$  a.e.\ in $\Omega$, and $\|\psi_j\|_\infty\leq \|u\|_\infty$,
$\|D\psi_j\|_\infty \leq 3 \|Du\|_\infty$ for every $j\in\N$.
\end{lemma}

\begin{proof}
We use a standard truncation argument.
Let $G\in C^1(\R)$ such that $0\leq G(t)\leq |t|$ for every $t\in\R$, $G(t) = 0$ for $|t|\leq 1$,
$G(t) = |t|$ for $|t|\geq 2$, $|G'(t)|\leq 3$ for every $t\in\R$.

Let us define $u_j := G(ju) / j$,
so that $u_j \to u$ pointwise in $\overline{\Omega}$.
Moreover, $Du_j = G'(ju) Du$, hence $Du_j\to Du$  a.e.\ in $\Omega$ and
$|Du_j| \leq 3 |Du|$.

Let $\varphi_{\epsilon}$ be the standard family of mollifiers in $\R^n$.
Since $u_j$ has compact support in $\Omega$,
we can choose a sequence $\varepsilon_j \searrow 0$,
such that the sequence $\psi_j := \varphi_{\epsilon_j} \ast u_j$
has the required properties.
\end{proof}

%

\begin{cor}\label{c:density}
$C^{\infty}_0(\Omega)$ is dense in the space
$X := \{u-w:\ u,w\in \spaceX\}$
with respect to the weak${}^*$ topology of $W^{1,\infty}(\Omega)$.
As a consequence,
if $(u,v)\in \spaceX\times L^1_+(\Omega)$ satisfies (\ref{f:weak}),
then
\begin{equation}\label{f:weakB}
\int_{\Omega} v \pscal{D\gauge(Du)}{Dw-Du} = \int_{\Omega} f(w-u)
\qquad \forall w\in \spaceX.
\end{equation}
\end{cor}

We are now in a position to prove that the Monge-Kantorovich system (\ref{f:MK})
is the Euler-Lagrange condition for the minimum problem (\ref{f:min}).

\begin{teor}\label{t:equiv}
The minimum problem (\ref{f:min}) and the system of PDEs (\ref{f:MK})
are equivalent in the following sense.
\begin{itemize}
\item[(i)]
$u\in\spaceX$ is a solution to (\ref{f:min}) if and only if
there exists $v \in L^1_+(\Omega)$
such that
$(u,v)$ is a solution to~(\ref{f:MK}).
\item[(ii)]
Let $u\in\spaceX$ be a solution to~(\ref{f:min}).
Then $(u,v)$ is a solution to~(\ref{f:MK})
if and only if $(\uo, v)$ is a solution to~(\ref{f:MK}).
\end{itemize}
\end{teor}

\begin{proof}
Let us denote by $I_{K}$ the indicator function of the set $K$, that is
\[
I_K(x) :=
\begin{cases}
0 & \textrm{if} \ x\in K\,, \\
+\infty & \textrm{if} \ x\in \R^n \setminus K\,,
\end{cases}
\]
so that
\[
F(u):=\int_{\Omega}[I_K(Du)-fu]\, dx = -\int_{\Omega}fu\, dx,
\qquad \forall u\in \spaceX\,.
\]
Since $\gauge$ is differentiable in $\R^n \setminus \{0\}$,
the subgradient of the indicator function
can be explicitly computed, obtaining
\begin{equation}\label{subindk}
\partial I_{K}(\xi)=
\begin{cases}
\{\alpha D\gauge(\xi)\colon \alpha \geq 0\} & \text{if}\ \xi\in \partial K, \\
\emptyset & \text{if}\ \xi\not\in K, \\
\{0\} & \text{if}\ \xi\in \inte K
\end{cases}
\end{equation}
(see e.g.\ \cite[Sect.~23]{Rock}).

Hence, if $(u,v)$ is a solution to (\ref{f:MK}),
then $v(x) D\gauge(Du(x)) \in \partial I_{K}(Du(x))$ for a.e.\ $x\in\Omega$, so that,
for every $w\in \spaceX$
\begin{equation}\label{f:minprf}
F(w) - F(u)\geq
\int_\Omega v \pscal{D\gauge(Du)}{Dw-Du}\, dx  -
\int_\Omega f\,(w-u)\, dx =0\,,
\end{equation}
where the last equality follows from
Corollary~\ref{c:density}.
This proves that $u$ is a solution to (\ref{f:min}).

Assume now that $u\in\spaceX$ is a minimizer for $F$, so that $f(u-\uo)=0$
a.e.\ in $\Omega$, due to the maximality of $\uo$ in $\spaceX$ and the fact that $f\geq 0$.
By Corollary~\ref{c:density}
we can choose $u-\uo$ as test function in the transport equation (\ref{f:dive})
solved by $\uo$, getting
\[
0=\int_\Omega \vf \pscal{D\gauge(D\uo)}{Du-D\uo}\, dx =
-\int_\Omega \vf (1-\pscal{D\gauge(D\uo)}{Du})\, dx\,.
\]
On the other hand, by Theorem \ref{t:sch}(ii)
and the fact that $Du\in K$ a.e.\ in $\Omega$, we have
\[
\begin{split}
1-\pscal{D\gauge(D\uo(x))}{Du(x)}& \geq 0,\ \text{a.e.}\ x\in \Omega\,, \\
1-\pscal{D\gauge(D\uo(x))}{Du(x)}& = 0 \
\Longleftrightarrow\ D\gauge(D\uo(x))=D\gauge(Du(x)),
\end{split}
\]
so that
$\vf D\gauge(D\uo)=\vf D\gauge(Du)$
and $\vf (1-\gauge(Du))=0$ a.e.\ in $\Omega$,
that is $(u, \vf)$ is a solution of (\ref{f:MK}).
This concludes the proof of (i).

Let us prove (ii).
The previous computation shows that if $(\uo,v)$ is a solution of (\ref{f:MK}), and $u\in\spaceX$ is a solution of the minimum problem (\ref{f:min}),
then also $(u,v)$ is a solution of (\ref{f:MK}).
Finally, let $(u,v)$ be a solution to (\ref{f:MK}).
Upon observing that $v\gauge(Du)=v$ a.e.\ in $\Omega$,
and choosing $u-\uo$
as test function in the weak formulation of the transport equation
\[
-\dive(v\, D\gauge(Du)) = f,
\qquad\textrm{in}\ \Omega
\]
we conclude that the opposite implication holds.
\end{proof}

\section{Uniqueness of the $v$--component}
\label{s:unique}

Thanks to Theorem~\ref{t:equiv}(ii), the uniqueness of the $v$ component
of the solution of the system of PDEs (\ref{f:MK}) will follow from the fact that the
function $v_f$ introduced in Theorem \ref{t:exv} is the unique weak solution
in $L^1_+(\Omega)$
of the equation
\begin{equation}\label{f:divei}
-\dive(v\, D\gauge(D\uo))=f\,.
\end{equation}

We shall use an explicit representation of $v_f$ proved in \cite{Bi}
and, in order to explain it,
we need to introduce some additional tools, related to the directions where the function $\uo$
has the maximal slope.


Recalling the representation formula (\ref{f:LH2}) for $\uo$,
for every $x\in\overline{\Omega}$ let us define the projections of $x$ by
\[
\Pi(x):=\{y\in V_x\colon\ \uo(x)=\phi(y)+\pgauge(x-y)\}
\]
and, for every $x\in\Omega$, let $\Delta(x)$ be the set of directions through $x$
\[
\Delta(x):=\left\{\frac{x-y}{\pgauge(x-y)}:\ y\in\Pi(x)\right\}\,, \qquad x\in\Omega\,.
\]
Let $D\subset\Omega$ be the set of points with multiple projections, that is
\begin{equation}\label{f:D}
D:=\{x\in \Omega:\ \Delta(x)\ \text{is not a singleton}\},
\end{equation}
and for every $x\in \Omega\setminus D$,
let $p(x)$ and $d(x)$ denote
the unique elements in $\Pi(x)$ and $\Delta(x)$ respectively.
Finally,
let $\tau(x)$ be defined by
\begin{equation}\label{f:tau}
\tau(x):=
\begin{cases}
\sup\{t\geq 0;\ \uo(x+sd(x)))=\uo(x)+s\ \forall s\in [0,t]\}, &
x\in \Omega\setminus D\,, \\
0 & x\in D\,,
\end{cases}
\end{equation}
and let $J$ be the set
\begin{equation}\label{f:qj}
J:=\bigcup_{x\in\Omega}q(x),\quad q(x):=x+\tau(x)d(x),
\end{equation}
where we understand that $q(x)=x$ if $x\in D$.
\begin{dhef}
We shall call \textsl{transport ray} through $x\in\Omega$
any segment $\cray{p,q(x)}$, $p\in\proj(x)$.
\end{dhef}
It is clear from the definition that, if $x\in \Omega\setminus D$, then
there is a unique transport ray $\cray{p(x),q(x)}$ through $x$.
On the other hand, if $x\in D$, any segment $\cray{p,x}$, with
$p\in\proj(x)$, is a transport ray through $x$.

The transport rays correspond to the segments where $\uo$ grows
linearly with maximal slope, as stated in the following lemma.
Hence in the sandpiles model, they correspond to the
directions where the matter is allowed to run down, that is where the $v$--component
could be nonzero.

\begin{lemma}\label{l:prunica}
For every $x\in\Omega$ and $y\in\proj(x)$ we have
\begin{equation}\label{f:ulin}
\uo(z) = \phi(y) + \pgauge(z-y)
\qquad \forall z\in\cray{y,x}\,.
\end{equation}
Moreover, if $x\in\Omega\setminus D$, then
\begin{equation}\label{f:ulin2}
\uo(z) = \phi(p(x)) + \pgauge(z-p(x))
\qquad \forall z\in\cray{p(x),q(x)}\,.
\end{equation}
\end{lemma}

\begin{proof}
The identity (\ref{f:ulin}) follows from (\ref{f:linpsi}), while
(\ref{f:ulin2}) for $x\in\Omega\setminus D$ follows from
(\ref{f:ulin}) when $z\in \cray{p(x),x}$, and from the very
definition of $q(x)$ when $z\in\cray{x, q(x)}$.
\end{proof}

%

\begin{cor}\label{l:project}
Given $x\in \Omega\setminus D$, let $\cray{p(x),q(x)}$ be the transport ray through $x$.
Then $\proj(z)=\{p(x)\}$ for every
$z \in \oray{p(x),q(x)}$.
\end{cor}

\begin{proof}
{}From (\ref{f:ulin2}) it is clear that
$p(x)\in\proj(z)$ for every $z\in\oray{p(x),q(x)}$.
Assume by contradiction that there exists $z \in \oray{p(x),q(x)}$
such that $\proj(z)$ contains a point $y \neq p(x)$.

We claim that the three points $y$, $z$ and $q(x)$ cannot be aligned.
Namely, we cannot have $z\in \oray{y, q(x)}$, for otherwise we would have $y=p(x)$.
On the other hand, we cannot have $q(x)\in \cray{y,z}$: indeed, in this case, by (\ref{f:ulin2})
we would have
$\uo(q(x)) = \phi(y) + \pgauge(q(x)-y)$, while, by (\ref{f:ulin}), we know that
$\uo(q(x)) = \uo(z)+\pgauge(q(x)-z)$.
Putting together these information, we get
\[
\begin{aligned}
\uo(z) & = \phi(y)+\pgauge(z-y)=\phi(y)-\pgauge(z-q(x))+\pgauge(q(x)-y) \\
& = \uo(z)+\pgauge(q(x)-z)+\pgauge(z-q(x))> \uo(z),
\end{aligned}
\]
a contradiction.

Since the curve
$\cray{y,z}\cup\cray{z,q(x)}$, joining $y$ to $q(x)$ and lying in $\Omega$
except for its endpoints, is not a segment,
by Lemma~\ref{l:geoint} it cannot be a geodesic,
so that
\[
\uo(q(x))< \phi(y)+\pgauge(z-y) + \pgauge(q(x)-z).
\]
Taking into account (\ref{f:ulin2}) we get
\[
\uo(z) = \uo(q(x)) - \pgauge(q(x)-z)
< \phi(y) + \pgauge(z-y),
\]
a contradiction with $y\in\proj(z)$.
\end{proof}

At this stage we are dealing with
three meaningful sets of  singular points associated to $\uo$:
\begin{itemize}
\item[$\Sigma$,] the sets of points in $\Omega$ where $\uo$ is not differentiable;
\item[$D$,] the sets of points in $\Omega$ with multiple projections;
\item[$J$,] the set of endpoints in $\Omega$ of the transport rays.
\end{itemize}

The set $J$ plays a central role in optimal transport theory.
As a matter of fact,
the disintegration of the Lebesgue measure in the reference set $\Omega$
along transport rays,
can be achieved only if $J$ has Lebesgue measure zero,
and this is one of the main tools of the theory.

If $\Omega$ and $K$ are sufficiently smooth (e.g.\ both of class $C^2$, and $K$ with strictly positive curvatures),
one can prove that $D$ coincides with $\Sigma$ (see, e.g. \cite{CaSi,CMf}),
so that, by Rademacher's Theorem,
$D$ has zero Lebesgue measure.
Moreover, using the  methods of
Differential Geometry, it is possible to prove that $J=\overline{\Sigma}$ and to characterize
$J\setminus \Sigma$ in terms of optimal focal points. As a consequence one can prove that
$J$ has zero Lebesgue measure and give detailed rectifiability results (see, e.g., \cite{CMf,EH}).

If $\Omega$ and $K$ are not regular, the set $\overline{\Sigma}$ may have positive Lebesgue measure
(see \cite{MM} for an example where $\Omega\in C^{1,\alpha}$, $\alpha<1$,
and $K$ is the unit ball),
and $J$ can be a proper subset of $\overline{\Sigma}$.

We recall here some known properties of the singular sets,
adding some remarks in order
to complete the description of the relationships between these sets.
The known results, proved in \cite{Bi}, are collected in the following
proposition.
\begin{prop}\label{p:D}
The sets $\Sigma$, $D$ and $J$ satisfy the following properties.
\begin{itemize}
\item[(i)] $D\subseteq \Sigma$. More precisely, if $\uo$ is differentiable at $x\in\Omega$, then $x$
has a unique projection and $\Delta(x)=\{D\gauge(D\uo(x))\}$.
\item[(ii)] The set $J$ has zero Lebesgue measure.
\item[(iii)] $D$ is a $(n-1)$ rectifiable set.
\end{itemize}
\end{prop}

\begin{rk}\label{r:exa}
We recall that $D=\Sigma$ if $K$ is strictly convex. Proposition \ref{p:D}(i) states that
$D\subseteq \Sigma$, in general.
The following example shows that it can happen that $D \neq \Sigma$.
Let us consider the set
\[
K^0=\left\{(x,y)\in\R^2\colon \left(|x|+\frac{\sqrt{2}}{2}\right)^2+y^2\leq 1\right\}\,
\]
and let $K$ be its polar set.
The set $K^0$ is strictly convex but it is not of class $C^1$,
hence
the set $K$ is of class $C^1$ but it is not strictly convex (see Theorem \ref{t:sch}).
Let $\Omega$ be the hexagon
\[
\Omega=\left\{(x,y)\in\R^2\colon |y|<2-|x|,\ |y|<1 \right\}
\]
(see Figure~\ref{fig:exa}).
The Lax--Hopf function $u_0$ associated to $K$ and to the boundary datum
$\phi=0$ is the distance function to the boundary
of $\Omega$ with respect to the convex metric $\pgauge$,
which can be explicitly computed as
\[
u_0(x)=
\begin{cases}
\sqrt{2}(1-|y|), & \text{if}\ (x,y)\in \Omega\cap\{|x|\leq 1\}, \\
\sqrt{2}(2-|x|-|y|), &  \text{if}\  (x,y)\in \Omega\cap\{|x|> 1\},
\end{cases}
\]
so that
$\Sigma \setminus D = \{(x,y)\in \Omega\colon\ |x|=1,\ y\neq 0\}$.
\end{rk}

\begin{figure}
\includegraphics[height=4cm]{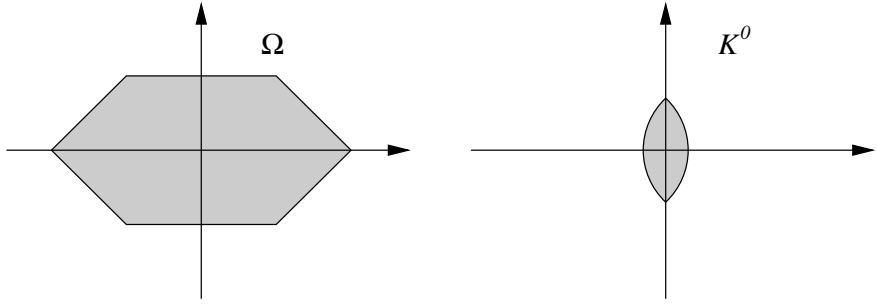}\qquad\qquad
\caption{The sets $\Omega$ and $K^0$ of Example~\ref{r:exa}}
\label{fig:exa}
\end{figure}


\begin{rk}
Since $D$ is a subset of $J$,
then by Proposition~\ref{p:D}(ii)
it has Lebesgue measure zero.
Notice that $D$ can be a proper subset of $J$
even in the regular case.
For example,
let $\Omega$ be the ellipsis $\{(x,y)\in\R^2;\ x^2/a^2+y^2/b^2 < 1\}$,
where $0< b < a$, and let $K=\overline{B}_1(0)$.
The points $P = ((a^2-b^2)/a, 0)$ and
$Q = (-(a^2-b^2)/a, 0)$ are the centers of curvature
of $\partial\Omega$ at $(a,0)$ and $(-a,0)$ respectively.
It can be checked that
$J=\cray{Q, P}$ whereas $D=\oray{Q, P}$.

Moreover, since $J$ need not be closed (see \cite[Remark 5.2]{Bi}),
$J\neq \overline{D}$ in general.
The following proposition shows that $J\subseteq \overline{D}$.
\end{rk}

\begin{prop}\label{p:BU}
For every $x\in \Omega \setminus \overline{D}$
we have that $\tau(x)>0$,
where $\tau$ is the function defined in (\ref{f:tau}).
\end{prop}

\begin{proof}
Let us fix $x \in \Omega \setminus \overline{D}$.
We can assume without loss of generality that $x=0$,
and $d(0)=-\lambda e_n$, $\lambda>0$.
Since the map $z \mapsto d(z)$ is continuous in the open set $\Omega \setminus \overline{D}$
(see \cite[Prop.~3.2]{Bi}), there exists $r>0$ such that
$\overline{B}_r(0)\subseteq \Omega \setminus \overline{D}$,
and  $\pscal{d(z)}{e_n}<-\lambda/2$ for every $z\in \overline{B}_r(0)$.

Given $\delta\in (0,r/2)$,
let us define the function $\psi\colon \overline{B}_r(0) \to \R^n$ by
\begin{equation}\label{f:bu}
\psi(z) := z + \left(\frac{2\delta-\pscal{z}{e_n}}{\pscal{d(z)}{e_n}}\right)d(z)\,,
\end{equation}
so that $\pscal{\psi(z)}{e_n}=2\delta$.
Hence $\psi$ maps $\overline{B}_r(0)$ into the hyperplane
$H := \{x\in\R^n: \pscal{x}{e_n} = 2\delta\}$.
The continuity of $d(z)$ in $\overline{B}_r(0)$
ensures that we can choose $\delta$ in such a way that
$\psi(z)\in \overline{B}_r(0)$ for every $z \in \overline{B}_\delta(0)$.

Summarizing, $\psi\colon \overline{B}_r(0)\to H$ is a continuous map,
$\psi(\overline{B}_{\delta}(0)) \subset \overline{B}_r(0)$,
and $\psi(z)$ belongs to the ray $\{z-t d(z): t\geq 0\}$.
Moreover, for every $z\in \overline{B}_{\delta}(0)$,
we have that $\cray{\psi(z),z}\subset \overline{B}_r(0)\subset\Omega$,
hence $\psi(z)\in \oray{p(z),z}$ and,
by Corollary~\ref{l:project}, we infer that $d(z)= d(\psi(z))$.

Since $\psi$ maps $\partial B_{\delta}(0)$
into the $(n-1)$-dimensional hyperplane $H$,
by the Borsuk--Ulam lemma (see \cite[Cor.~4.2]{Dei})
there exists $w\in \partial B_{\delta}(0)$ such that
$\psi(w)=\psi(-w)$.
Furthermore
$d(w) = d(\psi(w)) = d(\psi(-w)) = d(-w)$, so that
the equality $\psi(w)=\psi(-w)$ implies that
\[
w = \frac{\pscal{w}{e_n}}{\pscal{d(w)}{e_n}}\, d(w),
\]
i.e., $w$ is parallel to $d(w)$.
But the transport ray through $w$ and $-w$ contains the origin, hence
$d(w) = d(-w) = d(0)$ and $\tau(0) \geq \delta$.
%
%
\end{proof}

For the reader's convenience we collect here the results
concerning the relationships between the singular sets.

\begin{cor}
The sets $\Sigma$, $D$ and $J$ have zero Lebesgue measure.
Moreover, $D\subset\Sigma$ and $D\subset J\subset\overline{D}$,
with possibly strict inclusions.
\end{cor}

We are now in a position to briefly recall the explicit
representation of the function $v_f$ appearing in (\ref{f:dive}),
one of the main results in \cite{Bi}.

\begin{teor}\label{f:bian}
There exists a positive function $\alpha\colon\Omega\to (0,+\infty)$,
absolutely continuous along almost every transport ray,
such that an explicit solution $\vf\in L^1_+(\Omega)$ of~(\ref{f:dive})
is given by
\begin{equation}\label{f:vf}
\vf(x)=
\int_0^{\tau(x)} f(x+t d(x))\, \frac{\alpha(x+t d(x))}{\alpha(x)}\, dt,
\qquad\text{a.e.}\ x\in \Omega\,.
\end{equation}
More precisely, for a.e.\ $x\in\Omega\setminus D$ the function $\vf$
is locally absolutely continuous along the ray
$t\mapsto x+t d(x)$, $t\in [0,\tau(x))$, and satisfies
\begin{equation}\label{f:odev}
\begin{cases}
\displaystyle
\frac{d}{dt}\left[
\vf(x+t d(x))\, \alpha(x+t d(x))\right]
= - f(x+t d(x))\, \alpha(x+t d(x)),\\
\displaystyle
\lim_{t\to \tau(x)^-} \vf(x+t d(x))\, \alpha(x+t d(x)) = 0.
\end{cases}
\end{equation}
\end{teor}

\begin{proof}
See \cite{Bi}, Theorems~5.7 and~7.1.
\end{proof}

\begin{rk}
The reader may want to compare the representation formula~(\ref{f:vf})
with the one proved in the regular homogeneous case
(see \cite{CaCa,CCCG,CMf}),
where $\alpha(x)$ can be explicitly written in terms of the
distance from the boundary $u_0(x)$ and of the anisotropic
curvatures $\kappa_1,\ldots,\kappa_{n-1}$ of $\partial\Omega$
at $p(x)$.
More precisely, one has
\[
\alpha(x) = \prod_{j=1}^{n-1}
\left[1-u_0(x)\, \kappa_j(p(x))\right],
\qquad x\in\Omega\setminus D\,.
\]
\end{rk}


We shall find a necessary and sufficient condition for having that $\vf$ is the unique solution
to (\ref{f:divei}). The following example shows why the uniqueness of the $v$-component may fail.

\begin{ehse}\label{e:square}
Let $\Omega=(0,1)\times(0,1) \subseteq \R^2$, $K=\overline{B}_1(0)$,
and let $\phi(x,y)=y$ on $\partial \Omega$, so that
$\uo(x,y)=y$ in $\Omega$.
Then
every function
of the form $v(x,y)=1-y+c(x)$, $c(x)\geq 0$ for $x\in[0,1]$,
is a non-negative solution of
the equation $-\dive(v\, D\uo)=1$.
Notice that in this case $\vf(x,y)=1-y$ in $\Omega$,
so that~(\ref{f:divei}) has infinitely
many solutions of the form $v=\vf+\tilde{v}$,
with $\dive(\tilde{v}\, D\uo)=0$.
The existence of non-trivial variations of $\vf$
is due to the fact that $\Omega$ is covered by transport rays
with both endpoints on $\partial\Omega$.
\end{ehse}

Now it should be clear that, in order to discuss the uniqueness of
the $v$--component of the solutions to~(\ref{f:divei}),
we need to introduce the set
\[
T := \bigcup_{(p,q)\in E} \oray{p,q}\,,
\]
where the set $E$ is defined by
\[
E := \{(p,q)\in \partial \Omega \times \partial \Omega;\ p\neq q,\ \oray{p,q}\subseteq \Omega,\
\phi(q)=\phi(p)+\pgauge(q-p)\}.
\]

In Example~\ref{e:square}, $E=\{((x,0),(x,1)),\ x\in (0,1)\}$ and $T=\Omega$.
In general $T$ is the
union of the transport rays ending on $\partial\Omega$. On this set every function in $\spaceX$ is forced
to have maximal slope and hence to coincide with $\uo$.
\begin{lemma}\label{l:sinray}
Let $(p,q)\in E$ be given. Then the following hold.
\begin{itemize}
\item[i)] $\proj(x)=\{p\}$ and $q(x)=q$ for every $x\in \oray{p,q}$.
\item[ii)]
If $u\in \spaceX$,
then $u(x)=\phi(p)+\pgauge(x-p)$ for every $x\in\cray{p,q}$.
As a consequence every $u\in\spaceX$ coincides with $\uo$ on the set $T$.
\end{itemize}
\end{lemma}
\begin{proof}
Let us start by proving that $p\in\proj(x)$ for every $x\in \oray{p,q}$. Namely
\[
\begin{split}
\uo(x) & \leq \phi(p)+\pgauge(x-p)=\phi(q)-\pgauge(q-p)+\pgauge(x-p) \\
& =\phi(q)-\pgauge(q-x)\leq \uo(x)\,,
\end{split}
\]
where the last inequality follows from Lemma~\ref{l:diseqrho}.
Hence $\uo(x)=\phi(p)+\pgauge(x-p)$, so that $p\in\proj(x)$.

Actually we can prove that $\Pi(x)=\{p\}$ for every $x\in \oray{p,q}$
using the same argument of the proof of Corollary~\ref{l:project}:
if there exists $p_1\neq p$, $p_1\in \Pi(x)$,
then the three points $p_1$, $x$ and $q$ cannot be aligned, so that
\[
\phi(q) - \phi(p_1) < \pgauge(q-x) + \pgauge(x-p_1)
\]
and
\[
\uo(x) = \phi(p_1)+\pgauge(x-p_1) > \phi(q)-\pgauge(q-x)
= \phi(p) + \pgauge(x-p) = \uo(x),
\]
a contradiction.

Since $q\in\partial\Omega$ and, for every
$x\in\oray{p,q}$, $\proj(x) = \{p\}$, it is clear that $q(x) = q$.

Finally, let $u \in \spaceX$ and $x\in\cray{p,q}$ be given. Then,
using Lemma~\ref{l:diseqrho} we get
\[
\phi(p)+\pgauge(x-p)=\phi(q)-\pgauge(q-x) \leq u(x)\leq \phi(p)+\pgauge(x-p)\,,
\]
and (ii) follows.
\end{proof}

\begin{teor}\label{t:uniqv}
The function $\vf$ defined in ~(\ref{f:vf})
is the unique solution of (\ref{f:divei}) if and only if the set
$T$ has zero Lebesgue measure.
\end{teor}

\begin{proof}
The uniqueness of $\vf$ when $T$ has zero Lebesgue measure
is proved in \cite[Prop.~7.3]{Bi}.
The proof is based on the fact that every weak solution
$v\in L^1_+(\Omega)$ to~(\ref{f:divei})
is locally absolutely continuous along almost every transport ray,
and satisfies
\[
\frac{d}{dt}\left[
v(x+t d(x))\, \alpha(x+t d(x))\right]
= - f(x+t d(x))\, \alpha(x+t d(x))\,.
\]
Moreover, if the endpoint $q(x)$ of the transport ray through $x$
belongs to $\Omega$, the $v\cdot\alpha$
satisfies the initial condition
\[
\lim_{t\to \tau(x)^-} \vf(x+t d(x))\, \alpha(x+t d(x)) = 0
\]
so that, by Theorem~\ref{f:bian}, $v=\vf$ along that ray.

Let us now assume that $T$ has positive Lebesgue measure.
It is clear that the function
\[
\wo(x) := \min \{-\phi(y)+\pgauge(y-x):\ y\in V_x\},\qquad x\in\overline{\Omega}\,,
\]
is the Lax-Hopf function corresponding to the
geometry induced by the convex set $-K$ and the boundary datum $-\phi$.
Assume that $\uo$ and $\wo$ are differentiable at a point $x\in\oray{p,q}$, with $(p,q)\in E$.
{}From Proposition~\ref{p:D} we have that
$D\uo(x) = (q-p) / \pgauge(q-p)$.
The same argument, used with $-\phi$ as boundary datum and the geometry induced by $-K$,
shows that
\[
D\gauge_-(D\wo(x)) = \frac{p-q}{\pgauge_-(p-q)} = - \frac{q-p}{\pgauge(q-p)} =
D\gauge(D\uo(x)),
\]
where $\gauge_-$ and $\pgauge_-$ denote respectively the gauge functions of the convex sets
$-K$ and $-K^0$.

Let us consider the indicator function of the set $T$,
\[
g(x) :=
\begin{cases}
1, & x\in T, \\
0 ,& x\in \Omega \setminus T.
\end{cases}
\]
By Theorem \ref{t:exv},
there exist two nonnegative functions $v_g^+$, $v_g^-$ such that
\[
-\dive(v_g^+D\gauge(D\uo))=g\,, \quad -\dive(v_g^-D\gauge_{-}(D\wo))=g,
\qquad \text{in\ }\Omega\,.
\]
Looking at the explicit representations of $v_g^+$ and $v_g^-$
as in~(\ref{f:vf}), we easily get
\[
v_g^+=v_g^-=0\ \text{a.e. in}\ \Omega\setminus T\,,\qquad
v_g^+>0,\ v_g^->0 \ \text{a.e. in}\ T\,,
\]
so that, for every $\varphi\in C^{\infty}_0(\Omega)$, we obtain
\[
\begin{split}
\int_\Omega & (v^+_g+v^-_g)\pscal{D\gauge(D\uo)}{D\varphi}\, dx  =
\int_T(v^+_g+v^-_g)\pscal{D\gauge(D\uo)}{D\varphi}\, dx \\
& =
\int_T v^+_g \pscal{D\gauge(D\uo)}{D\varphi}\, dx -
\int_T v^-_g \pscal{D\gauge_{-}(D \wo)}{D\varphi}\, dx \\
& =
\int_{\Omega} g\,\varphi\, dx - \int_{\Omega} g\,\varphi\, dx = 0\,.
\end{split}
\]
Hence for every $\lambda\geq 0$ the function $v_\lambda=\vf+\lambda(v^+_g+v^-_g)$ is a
non-negative weak solution to~(\ref{f:divei}).
\end{proof}

\begin{rk}
It can be easily checked that if the datum $\phi$ satisfies (H4) with strict inequality
holding for every $x\neq y$,
then the Lebesgue measure of $T$ is zero, so that $\vf$ is the unique $v$-component allowed in (\ref{f:MK}).
\end{rk}

\begin{rk}
If the Lebesgue measure of $\Omega\setminus T$ is zero,
then by Lemma~\ref{l:sinray} we have that $\spaceX=\{\uo\}$.
On the other hand, it can be easily checked that the converse
implication is also true.
In \cite{BeCe} (see also \cite{Bi}) it is proved that
either $\spaceX\neq \{\uo\}$ or there exists
a weak solution
$v\in L^1_{loc}(\Omega)$ to
\begin{equation}\label{f:homo}
\dive(v\, D\uo) = 0
\end{equation}
such that $v > 0$ a.e.\ in $\Omega$.
In the proof of Theorem~\ref{t:uniqv} we have shown that, as soon as $T$ has positive Lebesgue measure,
it is possible to construct a solution $v\in L^1_+(\Omega)$
to~(\ref{f:homo}) such that $v>0$ a.e.\ on $T$.
\end{rk}

\section{Uniqueness of the $u$-component}
\label{s:uniqueu}

In Theorem~\ref{t:equiv} we have shown that the set of admissible $u$-components
of the solutions to (\ref{f:MK}) coincides with the set of solutions to the
minimum problem (\ref{f:min}).

In this section we shall construct, using Lemma~\ref{l:genprop}, the minimal solution to
(\ref{f:min}), and we shall give a necessary and sufficient condition in order
to have that  this function equals the maximal solution $\uo$ in $\Omega$.

Let us define the function  $\uf\colon\overline{\Omega}\to\R$ by
\begin{equation}\label{f:uf2}
\uf(x) := \sup\{\uo(z)-\len{\gamma}:\ z\in\partial\Omega\cup\spt(f),\
\gamma\in\ccurve{x,z}\},
\end{equation}
where $\spt(f)$ is the
complement in $\overline{\Omega}$ of the union of all
relatively open subsets $A\subseteq\overline{\Omega}$
such that $f=0$ a.e.\ in $A$.

\begin{prop}\label{p:uf}
The function $\uf$ is characterized by
\begin{equation}\label{f:uf}
\uf(x) = \max\{\uo(z)-\pgauge(z-x):\ z\in W_x\},
\end{equation}
where $W_x := \{y\in\partial\Omega\cup\spt(f):\ \oray{y,x}\subset\Omega\setminus\spt(f)\}$,
and satisfies
\begin{itemize}
\item[(i)] $\uf\in\spaceX$;
\item[(ii)] $\uf(x)=\uo(x)$ for every $x\in\spt(f)\cup T$;
\item[(iii)]
given $u\in \spaceX$, we have $u=\uo$ on $\spt(f)$
if and only if $\uf\leq u \leq \uo$ in $\Omega$.
\end{itemize}
\end{prop}

\begin{proof}
The function $\uf$ is the function $\wp$ defined in (\ref{f:psi})
corresponding to $D=\overline{\Omega}$,
$C= \partial\Omega\cup\spt(f)$,
and $\psi = \uo$ on $C$.
Clearly the function $\psi$ satisfies (\ref{f:H4}) on $C$,
since $\uo$ satisfies that condition
on $\overline{\Omega}$.
Then the stated properties follow from
Lemmas~\ref{l:genprop}, \ref{l:vischar} and~\ref{l:sinray}.
\end{proof}

\begin{teor}
A function $u\in\spaceX$ is a solution to~(\ref{f:min})
if and only if $\uf\leq u\leq\uo$.
Moreover, the function $\uf$ coincides with $\uo$ in $\Omega$
(and hence $\uo$ is the unique solution to (\ref{f:min}))
if and only if $J\cap \Omega \subseteq \spt(f)$.
\end{teor}

\begin{proof}
The first assertion follows from Proposition~\ref{p:uf}(iii)
and from the fact that $u\in\spaceX$ is a solution
to (\ref{f:min}) if and only if
$u=\uo$ on $\spt(f)$.
Let us prove the uniqueness result.
Assume that $J\cap \Omega \subseteq \spt(f)$, so that $J\subset\partial\Omega\cup\spt(f)$
and, from Proposition~\ref{p:uf}(ii), $\uf = \uo$ on $J$.
Let $x\in \Omega \setminus J$ be given, and let $q(x)\in J$
be the endpoint of the ray through $x$. We have
\[
\uo(x)=\uo(q(x))-\pgauge(q(x)-x) =\uf(q(x))-\pgauge(q(x)-x)\leq \uf(x)\leq  \uo(x)\,,
\]
and hence $\uf(x)=\uo(x)$.

Assume now by contradiction that $\uf=\uo$ in $\Omega$ and that there exists
$x_0\in J\cap\Omega$ such that $x_0\not\in \spt(f)$.
Let $z\in W_{x_0}$ and
$y\in V_{x_0}$ be such that
\begin{gather}
\uf(x_0)=\uo(z)-\pgauge(z-x_0), \label{f:zdix} \\
\uo(x_0)=\uo(y)+\pgauge(x_0-y). \label{f:ydix}
\end{gather}
Notice that $y\neq z$, otherwise we get $-\pgauge(y-x_0)=\pgauge(x_0-y)$. Hence $x_0$,
$y$ and $z$ are three distinct points.
Moreover we have that $x_0\in\oray{y,z}$,
otherwise by Lemma~\ref{l:geoint} we should have
\[
\uo(z)-\uo(y) < \pgauge(z-x_0) + \pgauge(x_0-y)
\]
and
\[
\uo(x_0) = \uo(z) - \pgauge(z-x_0)
< \uo(y) + \pgauge(x_0-y) = \uo(x_0),
\]
a contradiction.

Finally, the fact that $x_0\in \oray{y,z}$ implies that
\[
\begin{split}
\uo(z) & = \uo(x_0) + \pgauge(z - x_0)
= \uo(y) + \pgauge(x_0-y) + \pgauge(z-x_0)
\\ & = \uo(y) + \pgauge(z-y),
\end{split}
\]
hence $\oray{y,z}$ is a transport ray through $x_0$,
so that $x_0\not\in J$.
\end{proof}


\begin{thebibliography}{10}

\bibitem{Amb}
{L.} Ambrosio, \emph{Lecture notes on optimal transport problems}, Mathematical
  Aspects of Evolving Interfaces, Lecture Notes in Math., vol. 1812,
  Springer-Verlag, Berlin/New York, 2003, pp.~1--52.

\bibitem{AmTi}
{L.} Ambrosio and {P.} Tilli, Topics on analysis in metric spaces, Oxford
  Lecture Series in Mathematics and its Applications, vol.~25, Oxford
  University Press, Oxford, 2004. \MR{2039660 (2004k:28001)}

\bibitem{Ar-b}
G.~Aronsson, \emph{Interpolation under a gradient bound}, J. Aust. Math. Soc.
  \textbf{87} (2009), 19--35.

\bibitem{AEW}
{G.} Aronsson, {L.\ C.} Evans, and {Y.} Wu, \emph{Fast/slow diffusion and
  growing sandpiles}, J.\ Differential Equations \textbf{131} (1996), no.~2,
  304--335. \MR{MR1419017 (97i:35068)}

\bibitem{BeCe}
{S.} Bertone and {A.} Cellina, \emph{On the existence of variations, possibly
  with pointwise gradient constraints}, ESAIM Control Optim. Calc. Var.
  \textbf{13} (2007), 331--342.

\bibitem{Bi}
{S.} Bianchini, \emph{On the {E}uler-{L}agrange equation for a variational
  problem}, Discrete Contin. Dyn. Syst. \textbf{17} (2007), 449--480.

\bibitem{BiGl}
{S.} Bianchini and {M.} Gloyer, \emph{On the {E}uler-{L}agrange equation for a
  variational problem: the general case {II}}, Math. Z. \textbf{265} (2010),
  no.~4, 889--923. \MR{2652541}

\bibitem{BoBu}
{G.} Bouchitt\'e and {G.} Buttazzo, \emph{Characterization of optimal shapes
  and masses through {M}onge-{K}antorovich equation}, J.\ Eur.\ Math.\ Soc.
  \textbf{3} (2001), 139--168.

\bibitem{BuBuIv}
{D.} Burago, {Yu.} Burago, and {S.} Ivanov, A course in metric geometry,
  Graduate Studies in Mathematics, vol.~33, American Mathematical Society,
  Providence, RI, 2001. \MR{1835418 (2002e:53053)}

\bibitem{CaCa}
{P.} Cannarsa and {P.} Cardaliaguet, \emph{Representation of equilibrium
  solutions to the table problem for growing sandpiles}, J.\ Eur.\ Math.\ Soc.\
  (JEMS) \textbf{6} (2004), 435--464.

\bibitem{CCCG}
{P.} Cannarsa, {P.} Cardaliaguet, {G.} Crasta, and {E.} Giorgieri, \emph{A
  Boundary Value Problem for a {PDE} Model in Mass Transfer Theory:
  Representation of Solutions and Applications}, Calc.\ Var.\ Partial
  Differential Equations \textbf{24} (2005), 431--457.

\bibitem{CCS}
{P.} Cannarsa, {P.} Cardaliaguet, and {C.} Sinestrari, \emph{On a differential
  model for growing sandpiles with non-regular sources}, Comm. Partial
  Differential Equations \textbf{34} (2009), no.~7-9, 656--675. \MR{MR2560296}

\bibitem{CaMS}
{P.} Cannarsa, {A.} Mennucci, and {C.} Sinestrari, \emph{Regularity results for
  solutions of a class of {H}amilton-{J}acobi equations}, Arch.\ Rational
  Mech.\ Anal. \textbf{140} (1997), 197--223.

\bibitem{CaSi}
{P.} Cannarsa and {C.} Sinestrari, Semiconcave functions, {H}amilton-{J}acobi
  equations and optimal control, Progress in Nonlinear Differential Equations
  and their Applications, vol.~58, Birkh\"auser, Boston, 2004.

\bibitem{Ces}
{L.} Cesari, Optimization---theory and applications, Applications of
  Mathematics (New York), vol.~17, Springer-Verlag, New York, 1983.

\bibitem{CMf}
{G.} Crasta and {A.} Malusa, \emph{The distance function from the boundary in a
  {M}inkowski space}, Trans.\ Amer.\ Math.\ Soc. \textbf{359} (2007),
  5725--5759.

\bibitem{CMg}
{G.} Crasta and {A.} Malusa, \emph{On a system of partial differential
  equations of {M}onge-{K}antorovich type}, J.\ Differential Equations
  \textbf{235} (2007), 484--509.

\bibitem{CMi}
{G.} Crasta and {A.} Malusa, \emph{A sharp uniqueness result for a class of
  variational problems solved by a distance function}, J.\ Differential
  Equations \textbf{243} (2007), 427--447.

\bibitem{CMh}
{G.} Crasta and {A.} Malusa, \emph{A variational approach to the macroscopic
  electrodynamics of anisotropic hard superconductors}, Arch.\ Rational Mech.\
  Anal. \textbf{192} (2009), 87--115.

\bibitem{Dei}
{K.} Deimling, Nonlinear Functional Analysis, Springer-Verlag, Berlin, 1985.

\bibitem{EH}
{W.D.} Evans and {D.J.} Harris, \emph{Sobolev embeddings for generalized ridged
  domains}, Proc.\ London Math.\ Soc. \textbf{54} (1987), 141--175.

\bibitem{Gromov}
{M.} Gromov, Metric structures for {R}iemannian and non-{R}iemannian spaces,
  english ed., Modern Birkh\"auser Classics, Birkh\"auser Boston Inc., Boston,
  MA, 2007, Based on the 1981 French original, With appendices by M. Katz, P.
  Pansu and S. Semmes, Translated from the French by Sean Michael Bates.
  \MR{MR2307192 (2007k:53049)}

\bibitem{HK}
{K.P.} Hadeler and {C.} Kuttler, \emph{Dynamical models for granular matter},
  Granular\ Matter \textbf{2} (1999), 9--18.

\bibitem{IT}
{J.} Itoh and {M.} Tanaka, \emph{The {L}ipschitz continuity of the distance
  function to the cut locus}, Trans.\ Amer.\ Math.\ Soc. \textbf{353} (2001),
  21--40.

\bibitem{LN}
{Y.Y.} Li and {L.} Nirenberg, \emph{The distance function to the boundary,
  {F}insler geometry and the singular set of viscosity solutions of some
  {H}amilton--{J}acobi equations}, Commun.\ Pure Appl.\ Math. \textbf{58}
  (2005), 85--146.

\bibitem{Li}
{P.L.} Lions, Generalized solutions of {H}amilton-{J}acobi equations, Pitman,
  Boston, 1982.

\bibitem{MM}
{C.} Mantegazza and {A.C.} Mennucci, \emph{Hamilton-{J}acobi equations and
  distance functions on {R}iemannian manifolds}, Appl.\ Math.\ Optim.
  \textbf{47} (2003), 1--25.

\bibitem{Pr}
{L.} Prigozhin, \emph{Variational model of sandpile growth}, European J.\
  Appl.\ Math. \textbf{7} (1996), 225--235.

\bibitem{Rock}
{R.T.} Rockafellar, Convex Analysis, Princeton Univ.\ Press, Princeton, NJ,
  1970.

\bibitem{Sch}
{R.} Schneider, Convex bodies: the {B}runn--{M}inkowski theory, Cambridge
  Univ.~Press, Cambridge, 1993.


\bibitem{Vil}
{C.} Villani, Optimal transport, Grundlehren der Mathematischen
  Wissenschaften [Fundamental Principles of Mathematical Sciences], vol. 338,
  Springer-Verlag, Berlin, 2009, Old and new. \MR{2459454 (2010f:49001)}

\end{thebibliography}

\providecommand{\bysame}{\leavevmode\hbox to3em{\hrulefill}\thinspace}
\providecommand{\MR}{\relax\ifhmode\unskip\space\fi MR }
\providecommand{\MRhref}[2]{%
  \href{http://www.ams.org/mathscinet-getitem?mr=#1}{#2}
}
\providecommand{\href}[2]{#2}

\end{document}